\documentclass[11pt]{amsart}

\usepackage{fullpage,graphicx,amsfonts,amssymb,amsmath,amsthm}
\usepackage[all]{xy}
\usepackage[left=.8in,top=.8in,bottom=.8in,right=.8in,letterpaper]{geometry} 
\usepackage{mathtools}
\usepackage{graphicx}
\graphicspath{ {images/} }
\usepackage{enumerate}
\usepackage{setspace}
\usepackage{amssymb}

\allowdisplaybreaks
\theoremstyle{plain} 
\newtheorem{theorem}    {Theorem}

\newtheorem{lemma}      [theorem]{Lemma}
\newtheorem{corollary}  [theorem]{Corollary}
\newtheorem{proposition}[theorem]{Proposition}

\theoremstyle{definition}
\newtheorem{definition} [theorem]{Definition}

\theoremstyle{remark}
\newtheorem{remark}              {Remark}

\numberwithin{equation}{section}
\numberwithin{theorem}{section}

\newcommand\norm[1]{\left\lVert#1\right\rVert}
\usepackage{url}

\begin{document}

\title{On the Comparison of Measures of Convex Bodies via Projections and Sections}

\author{Johannes Hosle}
\address{Department of Mathematics, University of California, Los Angeles, CA 90095}
\email{jhosle@ucla.edu}

\begin{abstract}
In this manuscript, we study the inequalities between measures of convex bodies implied by comparison of their projections and sections. Recently, Giannopoulos and Koldobsky proved that if convex bodies $K, L$ satisfy $|K|\theta^{\perp}| \le |L \cap \theta^{\perp}|$ for all $\theta \in S^{n-1}$, then $|K| \le |L|$. Firstly, we study the reverse question: in particular, we show that if $K, L$ are origin-symmetric convex bodies in John's position with $|K \cap \theta^{\perp}| \le |L|\theta^{\perp}|$ for all $\theta \in S^{n-1}$ then $|K| \le \sqrt{n}|L|$. The condition we consider is weaker than both the conditions $|K \cap \theta^{\perp}| \le |L \cap \theta^{\perp}|$ and $|K|\theta^{\perp}| \le |L|\theta^{\perp}|$ for all $\theta \in S^{n-1}$ that appear in the Busemann-Petty and Shephard problems respectively. Secondly, we appropriately extend the result of Giannopoulos and Koldobsky to various classes of measures possessing concavity properties, including log-concave measures.
\end{abstract}
\maketitle
\begin{section}{Introduction}
Understanding relations between convex bodies based on relations of their lower-dimensional measurements has long been of interest in the area of geometric tomography (see e.g. Gardner \cite{gardnertomography}). For example, the Busemann-Petty problem, posed in \cite{busemannpetty}, is the following question:

\textit{If $K$ and $L$ are origin-symmetric convex bodies in $\mathbb{R}^n$ such that \begin{align*}
    |K \cap \theta^{\perp}| \le |L \cap \theta^{\perp}|
\end{align*} for all $\theta \in S^{n-1}$, does it follow that $|K| \le |L|$?} 

Lutwak \cite{lutwak1988} showed that the Busemann-Petty problem is equivalent to the statement that all origin-symmetric convex bodies are intersection bodies, a notion that he introduced in that paper. Papadimitrakis \cite{papa}, Gardner \cite{Gardner1} and Zhang \cite{Zhang1} disproved the conjecture for dimensions $5$ and higher, where Gardner and Zhang used the equivalence of Lutwak. Gardner \cite{Gardner2} demonstrated an affirmative answer for $n=3$, while Zhang \cite{Zhang2} later verified the conjecture for $n=4$. For $n=2$, the assumption implies $K \subseteq L$, from which the inequality follows immediately. A uniform solution for all dimensions was accomplished by Gardner, Koldobsky, and Schlumprecht \cite{GKS}, where the authors applied the powerful techniques of the Fourier transform on the sphere to answer this question. 

Zvavitch \cite{zvavitch} demonstrated the solution to the generalized Busemann-Petty problem (with the same conclusions, affirmative for $n\le 4$ and negative for $n\ge 5$) for essentially arbitrary measures.

The Shephard problem, posed in \cite{shepard}, asked the analogous question about projections onto hyperplanes, rather than hyperplane sections:

\textit{If $K$ and $L$ are origin-symmetric convex bodies in $\mathbb{R}^n$ such that \begin{align*}
    |K|{\theta^{\perp}}| \le |L|{\theta^{\perp}}|
\end{align*} for all $\theta \in S^{n-1}$, does it follow that $|K| \le |L|$?}

As in the case of the Busemann-Petty problem, the answer to this question is negative in general. The claim only holds in dimensions $n=1,2$, as was demonstrated by Schneider \cite{schneider} and Petty \cite{petty}. For a comprehensive overview of the history of both problems, the reader is referred to Koldobsky \cite{koldbook}, where one can also find an introduction to the powerful Fourier analytic approach in convex geometry.

In view of the fact that $|K \cap \theta^{\perp}| \le |K|{\theta^{\perp}}|$ for any convex $K$ in $\mathbb{R}^n$ and $\theta \in S^{n-1}$, a natural variant of these problems with a more restrictive condition was posed by Milman:

\textit{If $K$ and $L$ are origin-symmetric convex bodies in $\mathbb{R}^n$ such that \begin{align*}
    |K|{\theta^{\perp}}| \le |L\cap {\theta^{\perp}}|
\end{align*} for all $\theta \in S^{n-1}$, does it follow that $|K| \le |L|$?}

Unlike the Busemann-Petty and Shephard problems, this was recently proved in the affirmative by Giannopoulos and Koldobsky \cite{giankold}. In fact, they showed a stronger statement:

\textit{If $K$ is a convex body in $\mathbb{R}^n$ and $L$ is a compact set in $\mathbb{R}^n$ such that, for some $1\le k \le n-1$, \begin{align*}
    |K|F| \le |L \cap F|
\end{align*} for all $k-$dimensional subspaces $F$ of $\mathbb{R}^n$, then $|K| \le |L|.$}


We will discuss both a variant of Milman's question and generalizations of Giannopoulos and Koldobsky's result for the case $k=n-1$. 

For Lebesgue measure, the notion of projection is unambiguous due to the measure's translation invariance. For general measures, we will use a generalization of this object, previously studied by Livshyts \cite{livshyts}. Recall that the mixed volume of $K$ and $L$ is defined by $V_1(K, L) = \frac{1}{n}\displaystyle \liminf_{\varepsilon \to 0}\frac{|K+\varepsilon L| - |K|}{\varepsilon}.$
We now recall the extension of this notion to measures.
\begin{definition}
\textit{Given a measure $\mu$ and measurable sets $A, B$, we define \begin{align*}
    \mu_1(A, B) := \liminf_{\varepsilon \to 0}\frac{\mu(A+\varepsilon B) - \mu(A)}{\varepsilon}
\end{align*} to be the mixed $\mu-$measure of $A$ and $B$.}
\end{definition}

Observe that the area of the projection $|K|\theta^{\perp}| = \frac{n}{2}V_1(K, [-\theta, \theta])$ for all $\theta \in S^{n-1}$. In view of this, it is natural that for a convex body $K$ and an absolutely continuous measure $\mu$, the $\mu-$projection be defined as \begin{align}P_{\mu, K}(\theta) = \frac{n}{2}\int_{0}^{1}\mu_1(tK, [-\theta, \theta])dt\end{align} for $\theta \in S^{n-1}$, which admits the geometric interpretation as an average of the mixed volumes of scalings of $K$ and the line segment $[-\theta, \theta]$. See Section 2 for further discussion.

\begin{subsection}{Reversal of Milman's Question}
A natural variant of Milman's question is to ask what we can conclude if for convex bodies $K, L$ we have \begin{align*}
	|K \cap \theta^{\perp}| \le |L|\theta^{\perp}|
\end{align*} for all $\theta \in S^{n-1}$. This condition is weaker than the conditions for both the Busemann-Petty and Shephard problems, and hence we cannot hope to conclude $|K| \le |L|$ in general. Milman and Pajor \cite{milmanpajor} demonstrated an isomorphic version of the Busemann-Petty problem, namely that if $K, L$ are origin-symmetric convex bodies with \begin{align*}
|K \cap \theta^{\perp}| \le |L\cap \theta^{\perp}|
\end{align*} for all $\theta \in S^{n-1}$, then $|K| \le cL_K |L|,$ where $L_K$ is the isotropic constant of $K$ and $c>0$ is an absolute constant. Koldobsky and Zvavitch \cite{kziso} proved an extension for measures, namely that if $\mu$ is a measure with an even continuous density, and $K, L$ are origin-symmetric convex bodies with \begin{align*}
    \mu(K \cap \theta^{\perp}) \le \mu(L \cap \theta^{\perp})
\end{align*} for all $\theta \in S^{n-1}$, then $\mu(K) \le \sqrt{n}\mu(L).$ Ball \cite{ballshadows} showed the corresponding result for Shephard's problem (though only for Lebesgue measure): If \begin{align*}
    |K|\theta^{\perp}| \le |L|\theta^{\perp}|
\end{align*} for all $\theta \in S^{n-1}$, then $|K| \le \frac{3}{2}\sqrt{n}|L|,$ a statement which he also shows to be sharp up to an absolute constant. In comparison to Milman and Pajor's result, Ball's result enjoys the benefit of a conclusion in terms of elementary functions, while determining the optimal bound on $L_K$ is one of the major unsolved questions in convex geometry (see the next section for more details).

Our main result addressing this variant problem of the sections of $K$ bounded by the projections of $L$ is dependent on the circumradius of $K$ and the inradius of $L$. In particular, we show the following:

\begin{theorem}
Let $K, L$ be convex bodies in $\mathbb{R}^n$ such that $K \subseteq RB_2^n$ and $rB_2^n \subseteq L$. \begin{enumerate}[(a)]
\item If $$|K \cap \theta^{\perp}| \le |L|\theta^{\perp}|$$ for all $\theta \in S^{n-1}$, then \begin{align*}
|K| \le \min\left(\frac{R}{r}, cL_K^{\frac{1}{2}}n^{\frac{3}{4}}\left(\frac{R}{r}\right)^{\frac{n}{2n-1}}\right)|L|,
\end{align*} where $c>0$ is some absolute constant.
\item If $L$ is in addition origin-symmetric and $\mu$ is an arbitrary absolutely continuous measure such that $$\mu_{n-1}(K \cap \theta^{\perp}) \le P_{\mu, L}(\theta)$$ for all $\theta \in S^{n-1}$, then \begin{align*}
\mu(K) \le \frac{R}{r(1-\frac{1}{n})}\mu(L).
\end{align*}
\end{enumerate}
\end{theorem}

Here and throughout, $\mu_{n-1}(L \cap \theta^{\perp}) = \int_{L \cap \theta^{\perp}}g(x)d\lambda_{n-1}(x),$ where $g$ is the continuous density of $\mu$. 

A variant of this theorem for projections and sections onto subspaces of arbitrary dimension is given in Proposition 3.1 in Section 3. 

For any convex body $K$, there exists a unique ellipsoid $\mathcal{E} \subseteq K$ of maximal volume (see e.g. Proposition 2.1.6 in Artstein-Avidan, Giannopoulos, and Milman \cite{AGM}). A convex body is said to be in John's position if this maximal ellipsoid is the unit ball. Every convex body can be transformed into this position via an affine map.

John's theorem (e.g. Theorem 2.1.3 also in \cite{AGM}) states that for any origin-symmetric convex body $K$ in John's position we have $K \subseteq \sqrt{n}B_2^n.$ The following statement is thus a consequence of Theorem 1.2:

\begin{corollary} Let $K, L$ be origin-symmetric convex bodies in $\mathbb{R}^n$ in John's position. 
\begin{enumerate}[(a) ] \item If $$|K\cap \theta^{\perp}| \le |L|\theta^{\perp}|$$ for all $\theta \in S^{n-1}$, then \begin{align*}
|K| \le \sqrt{n}|L|. 
\end{align*}
\item If $\mu$ is an arbitrary absolutely continuous measure such that $$\mu_{n-1}(K \cap \theta^{\perp}) \le P_{\mu, L}(\theta)$$ for all $\theta \in S^{n-1}$, then $$\mu(K) \le \frac{\sqrt{n}}{1-\frac{1}{n}} \mu(L).$$
\end{enumerate}
\end{corollary}

\begin{remark}
Since projections bound sections, Corollary 1.3a implies that if $K, L$ are origin-symmetric convex bodies in John's position with either $|K \cap \theta^{\perp}| \le |L\cap \theta^{\perp}|$ or $|K|\theta^{\perp}| \le |L|\theta^{\perp}|$ for all $\theta \in S^{n-1}$, then $|K| \le \sqrt{n}|L|$. In this particular case of $K$ and $L$ in John's position, the result on comparison of projections is a slight improvement to Ball's more general result \cite{ballshadows} that $|K| \le \frac{3}{2}\sqrt{n}|L|$.
\end{remark}

\end{subsection}

\begin{subsection}{Extending the Result of Giannopoulos and Koldobsky}
We shall extend the result of Giannopoulos and Koldobsky to various contexts. The first is the case of measures with $p-$concave, $\frac{1}{p}-$homogeneous densities. Examples of such densities include constant densities and the functions $g_{\theta}(x) = 1_{\langle x,\theta \rangle > 0} \langle x,\theta\rangle^{\frac{1}{p}}$ for $\theta \in \mathbb{R}^n$. With the exception of constant densities, all such densities must be supported on convex cones. See Milman and Rotem \cite{milmanrotem} for more discussion. Recently, Livshyts \cite{livshyts} proved a version of Shephard's problem for measures of this type.

Our theorem for measures with $p-$concave, $\frac{1}{p}-$homogeneous densities is the following separation result:

\begin{theorem}
Let $\mu$ be a measure on $\mathbb{R}^n$ with a continuous $p-$concave, $\frac{1}{p}-$homogeneous density $g$ for some $p>0$. Assume that $K$ is an origin-symmetric convex body in $\mathbb{R}^n$ and $L$ is a star body in $\mathbb{R}^n$ such that \begin{align*}
    P_{\mu, K}(\theta)\le \mu_{n-1}(L \cap \theta^{\perp}) + \varepsilon
\end{align*} for all $\theta \in S^{n-1}$. For $q=\frac{1}{n+\frac{1}{p}}$, we have \begin{align*}
    \mu(K)^{1-q} \le \left(\frac{1-\frac{1}{n}}{1-q}\right)\mu(L)^{1-q} + \frac{\omega_n}{\mu(B_2^n)^q\omega_{n-1}}\varepsilon.
\end{align*}
\end{theorem}

\begin{remark}
The reason for studying the 'separation' result is that measures $\mu \neq c\lambda$ of the above form will be supported in a half plane to one side of a hyperplane $H$ through the origin. Thus $\mu_{n-1}(L \cap H) = 0$ and $\mu_{n-1}(L \cap \theta^{\perp})$ can be made arbitrarily small for $\theta$ approaching the appropriate normal vector to $H$.
\end{remark}

In Section 5, we also study the the problem without the condition of homogeneity. Using a generalization of Grinnberg's inequality from Dann, Paouris, and Pivovarov \cite{DPP} in combination with properties of $q-$concave measures, we prove Theorem 5.1, an analog of Theorem 1.4. Such an analog is then also proved for log-concave measures with ray-decreasing densities in Theorem 6.1 of Section 6.
\end{subsection}

\textbf{Acknowledgements. }I am very grateful to Professor Galyna Livshyts for introducing me to the subject, suggesting these problems to me, and for her guidance. I would like to also thank Professor Michael Lacey for very helpful conversations we had on the topics of the paper. I am furthermore grateful to the Georgia Institute of Technology for supporting my stay at their REU program, where most of this work was done. Lastly, I thank Professor Alexander Koldobsky for useful comments on this paper.

\end{section}

\begin{section}{Preliminaries and Technical Lemmas}
\begin{subsection}{General Terminology}
The Lebesgue measure $\lambda$ of a measurable set $K$ in $\mathbb{R}^n$ will be denoted by $|K|$, or occasionally by $|K|_n$ to reference the dimension. Throughout, measurable means Borel measurable. The unit ball in $\mathbb{R}^n$ is denoted by $B_2^n$, and its Lebesgue measure will be represented by $\omega_n$. Let us recall that \begin{align}\begin{split}
\omega_n &= \frac{\pi^{\frac{n}{2}}}{\Gamma\left(\frac{n}{2}+1\right)} \sim \frac{1}{\sqrt{\pi n}}\left(\frac{2\pi e}{n}\right)^{\frac{n}{2}}
\end{split}
\end{align} as $n\to \infty$. Here $f(n) \sim g(n)$ means $\lim_{n\to\infty}\frac{f(n)}{g(n)} = 1$, while $f \lesssim g$ will be used denote the existence of an absolute constant $C>0$ such that $f \le Cg$.

We will use $S^{n-1}$ to denote the unit sphere in $\mathbb{R}^n$. By $G_{n,k}$, we mean the space of $k-$dimensional subspaces of $\mathbb{R}^n$, and we will denote the Haar probability measure on $G_{n,k}$ by $\nu_{n,k}$.

A set $K$ in $\mathbb{R}^n$ is called convex if the interval joining any two points in $K$ is also contained in $K$. If $K$ is also compact and has non-empty interior, $K$ is then called a convex body. Its Minkowski functional will be defined as \begin{align*}
\norm{x}_K = \min \{a\ge 0: x \in aK\}
\end{align*} for $x \in \mathbb{R}^n$, and $\rho_K(x) = \norm{x}_K^{-1}$ will be the radial function. If $0 \in \text{int}(K)$, then for $\theta \in S^{n-1}$, $\rho_K(\theta)$ is the distance from the origin to $\partial K$ in the direction of $\theta$. Next, the support function $h_K$ of $K$ is defined by \begin{align*}
h_K(x) = \max_{\xi \in K}\langle x,\xi \rangle.
\end{align*} 

The Gauss map of $K$ is the map $\nu_K: \partial K \to S^{n-1}$ that sends $y \in \partial K$ to the set of normal vectors to $K$ at $y$. The surface area measure $S(K, \cdot)$ is then defined by $S(K, E) = H_{n-1}(\nu_K^{-1}(E))$ for all measurable $E \subseteq S^{n-1}$, where $H_{n-1}$ is the $(n-1)-$dimensional Hausdorff measure on $\mathbb{R}^n$.  If $S(K, E)$ is absolutely continuous with respect to $H_{n-1}|_{S^{n-1}}$, then the density of $S$ is called the curvature function of $K$ and is denoted by $f_K$.

Given a convex body $K$ in $\mathbb{R}^n$, the isotropic position of $K$ is defined as the (unique up to orthogonal transformations) affine image $\tilde{K}$ of $K$ with volume 1 and barycenter at the origin such that \begin{align*}
\int_{\tilde{K}}x_ix_j dx = L_K \delta_{ij}
\end{align*} for all $i, j \in \{1,..,n\}$ and some constant $L_K > 0$. Here $x = (x_1, ..., x_n)$ are the coordinates in $\mathbb{R}^n$ and $\delta_{ij}$ is the Kronecker delta symbol. If $K$ is in isotropic position to begin with, we call $K$ isotropic. We will call $L_K$ the isotropic constant of $K$ and set \begin{align*}
    L_n = \max\{L_K: K \text{ is a convex body in } \mathbb{R}^n\}.
\end{align*}A consequence of John's theorem is $L_n \lesssim \sqrt{n}$ (e.g. Proposition 10.1.9 in Artstein-Avidan, Giannopoulos, and Milman's \cite{AGM}). The best current upper bound on $L_n$ is due to Klartag \cite{klartagbound}, who showed $L_n \lesssim n^{\frac{1}{4}},$ removing a logarithmic factor that appeared in a previous result of Bourgain \cite{bourgain}. The celebrated slicing problem, originally proposed by Bourgain \cite{bourgainslicing}, is equivalent to the assertion that $L_n \lesssim 1$ (see Milman and Pajor \cite{milmanpajor} for a discussion of various equivalent formulations of this conjecture).

Finally, a star body is a compact set in $\mathbb{R}^n$ such that $x \in K \Rightarrow [0,x) \subseteq \text{int}(K)$, with $[0,x)$ denoting the segment joining $0$ and $x$.

\end{subsection}

\begin{subsection}{Projections for Arbitrary Measures}
Below, we state a generalized notion of the surface area measure, see e.g. Livshyts \cite{livshyts}.
\begin{definition}
\textit{Let $\mu$ be a measure on $\mathbb{R}^n$ with density $g$ continuous on its support. If $K$ is a convex body, define the surface area measure of $K$ with respect to $\mu$ to be the following measure on $S^{n-1}$: \begin{align*} 
    \sigma_{\mu, K}(\Omega) = \int_{\nu^{-1}_K(\Omega)}g(x)dH_{n-1}(x),
\end{align*} for every measurable $\Omega \subseteq S^{n-1}$.}
\end{definition}

We now derive a formula for the generalized notion of projection, already defined in (1.1) in the introduction. Let $\mu$ be a measure on $\mathbb{R}^n$ and $K$ be a convex body. For $\theta \in S^{n-1}, t \in [0,1]$ set \begin{align}
    p_{\mu,K}(\theta, t) = \frac{n}{2}\int_{S^{n-1}}|\langle \theta,v \rangle|d\sigma_{\mu, tK}(v).
\end{align} By Lemma 3.3 in Livshyts \cite{livshyts}, we write this as \begin{align*}
p_{\mu, K}(\theta) &= \frac{n}{2}\int_{S^{n-1}}h_{[-\theta,\theta]}(v) d\sigma_{\mu, tK}(v) \\ &= \frac{n}{2}\mu_1(tK, [-\theta, \theta]),
\end{align*} and therefore, in view of (1.1) - the definition of $P_{\mu, K}$ - we have
\begin{align}
    P_{\mu, K}(\theta) = \int_{0}^{1}p_{\mu,K}(\theta,t) dt.
\end{align}

The function $p_{\mu, K}(\theta,t)$ in (2.2) serves as a weighted projection of the boundary of $tK$ and $P_{\mu, K}$ in (2.3) serves as an appropriate average. For Lebesgue measure $\lambda$, we can confirm that $P_{\lambda, K}(\theta) = |K|{\theta^{\perp}}|$. Indeed, by Cauchy's projection formula, \begin{align*}
    p_{\lambda, K}(\theta,t) &= \frac{n}{2}\int_{S^{n-1}}|\langle \theta, v\rangle|dS(tK, v) \\ &= n |tK|{\theta^{\perp}}| \\ &= nt^{n-1}|K|{\theta^{\perp}}|,
\end{align*} and the conclusion follows by the integration $\int_{0}^{1}t^{n-1} dt = \frac{1}{n}$. 

\end{subsection}
\begin{subsection}{Extension of Measures.}

Following e.g. Koldobsky \cite{koldbook}, given a Borel measure $\mu$ on $S^{n-1}$, we consider an extension of it to a distribution $\mu_e$ with degree of homogeneity $-(n+1)$ by setting \begin{align}
    \langle \mu_e, \phi\rangle = \frac{1}{2}\int_{S^{n-1}}\langle r^{-2}, \phi(ru)\rangle d\mu(u).
\end{align} for all $\phi \in \mathcal{S}(\mathbb{R}^n),$ the class of Schwartz functions on $\mathbb{R}^n$. 

The following result is proved by Koldobsky, Ryabogin, and Zvavitch \cite{KRZ}, using Lemma 1 from Koldobsky \cite{inverseformula} and the connections between the Fourier and spherical Radon transforms (see e.g. Lemma 2.11 in Koldobsky \cite{koldbook}). For the reader's convenience, we include a full proof in the Appendix.
\begin{lemma}
Let $\mu$ be a Borel measure on $\mathbb{R}^n$ and $\mu_{e}$ be its extension to a distribution with degree of homogeneity $-(n+1)$. Then \begin{align*}
    \widehat{\mu_e}(\theta) = -\frac{\pi}{2} \int_{S^{n-1}}|\langle u, \theta \rangle | d\mu(u)
\end{align*}for all $\theta \in S^{n-1}.$
\end{lemma}
A direct consequence of Lemma 2.2 and (2.2) is the following identity: \begin{align}
    \widehat{\sigma_{\mu, tK}}(\theta) = - \frac{\pi}{n} p_{\mu, K}(\theta, t).
\end{align}
\end{subsection}

\begin{subsection}{Concavity and Homogeneity}

Let us recall the definitions of $p$ concavity and $r-$homogeneity. 
\begin{definition}
\textit{A function $f: \mathbb{R} \to [0,\infty]$ is $p-$concave for some $p \in \mathbb{R}\setminus \{0\}$ if for all $\lambda\in [0,1]$ and $x,y\in \text{supp}(f)$ we have the inequality \begin{align*}
    f(\lambda x + (1-\lambda)y) \ge \left(\lambda f^p(x) + (1-\lambda) f^p(y)\right)^{\frac{1}{p}}.
\end{align*}}
\end{definition}
\begin{definition}
\textit{A function $f: \mathbb{R} \to [0,\infty]$ is $r-$homogeneous if for all $a > 0, x\in\mathbb{R}^n$ we have $f(ax) = a^r f(x).$}
\end{definition}

Consider a measure $\mu$ that has a continuous density $g$ that is both $s-$concave for some $s > 0$ and $\frac{1}{p}-$homogeneous for some $p > 0$. Under these assumptions $g$ will also be $p$-concave (see e.g. Proposition .5 in Livshyts \cite{livshyts}). 

By a change of variables, a measure with an $r-$homogeneous density will be an $(n+r)-$homogeneous measure, that is $\mu(tE) = t^{n+r}\mu(E)$ for all $t>0$ and measurable $E$.

We quote below a result of Borell \cite{borell} on the relationship between the degree of concavity of the density and of the measure, which is a generalization of the Brunn-Minkowski inequality:

\begin{lemma}[Borell]
Let $p \in (-\frac{1}{n}, \infty]$ and let $\mu$ be a measure on $\mathbb{R}^n$ with $p-$concave density $g$. If $q = \frac{1}{n+\frac{1}{p}}$, then $\mu$ is a $q$-concave measure, that is for all measurable $E, F$ and $\lambda \in [0,1]$, \begin{align*}
    \mu(\lambda E + (1-\lambda) F) \ge \left(\lambda \mu(E)^q + (1-\lambda)\mu(F)^q\right)^{\frac{1}{q}}.
\end{align*}
\end{lemma}

The following inequality concerning mixed volume is a corollary of Lemma 2.5, we refer the reader to Milman and Rotem \cite{milmanrotem}:

\begin{lemma}
Let $\mu$ be a measure on $\mathbb{R}^n$ with a $p-$concave, $\frac{1}{p}-$homogeneous density $g$ and set $q = \frac{1}{n+\frac{1}{p}}$. Then for all measurable $E, F$ we have the inequality \begin{align*}
    \mu(E)^{1-q}\mu(F)^q \le q \mu_1(E, F).
\end{align*}
\end{lemma}

The next statement gives an expression of the measure of a star body using its radial function.

\begin{lemma}
Let $L$ be a star body in $\mathbb{R}^n$ and $\mu$ be a measure with a continuous $\frac{1}{p}-$homogeneous density $g$ for some $p>0$. If $q = \frac{1}{n+\frac{1}{p}}$, then \begin{align*}
    \mu(L) = q \int_{S^{n-1}}\rho_L^{\frac{1}{q}}(\theta) g(\theta) d\theta.
\end{align*}
\end{lemma}
\begin{proof}
By polar coordinates and homogeneity, \begin{align*}
    \mu(L) &= \int_{\mathbb{R}^n}\chi_{[0,1]}(\norm{x}_L)g(x) dx \\ &= \int_{S^{n-1}}\left(\int_{0}^{\infty}\chi_{[0,1]}(r\norm{\theta}_L)g(r\theta)r^{n-1}  dr\right) d\theta \\ &= \int_{S^{n-1}}\left(\int_{0}^{\frac{1}{\norm{\theta}_L}}r^{n+\frac{1}{p}-1}g(\theta)dr\right)d\theta \\ &= q \int_{S^{n-1}} \norm{\theta}_L^{-\frac{1}{q}}g(\theta)d\theta \\ &= q\int_{S^{n-1}}\rho_L^{\frac{1}{q}}(\theta) g(\theta) d\theta.
\end{align*}
\end{proof}
\end{subsection}
\begin{subsection}{Projection and Surface Area}
The surface area of a convex body is defined to be \begin{align*}
    |\partial K| = \liminf_{\varepsilon \to 0}\frac{|K + \varepsilon B_2^n| - |K|}{\varepsilon},
\end{align*} the mixed $\lambda-$measure of $K$ and the unit ball $B_2^n$ in $\mathbb{R}^n$. Cauchy's surface area formula states that \begin{align}
    |\partial K| = \frac{1}{\omega_{n-1}}\int_{S^{n-1}}|K|u^{\perp}| du.
\end{align}

We prove a generalization of this formula in the case of essentially arbitrary measures. The proof will require the following fact: \begin{lemma}
In $\mathbb{R}^n$, the Fourier transform of the distribution $|x|$ is equal to \begin{align*}
    \widehat{(|x|)}(\xi) = -\frac{(2\pi)^n\Gamma\left(\frac{n+1}{2}\right)}{\pi^{\frac{n+1}{2}}}|\xi|^{-n-1}.
\end{align*}
\end{lemma}
The proof is included in the Appendix for the reader's convenience. 

\begin{proposition}
\textit{Let $\mu$ be a measure with continuous density. Then, for any origin-symmetric convex body $K$, \begin{align*}
    \frac{1}{n\omega_{n-1}}\int_{S^{n-1}}P_{\mu, K}(u) du &= \int_{0}^{1}\mu_1(tK, B_2^n) dt,
\end{align*} where $P_{\mu,K}(u)$ was defined in (1.1).}
\end{proposition}
\begin{proof} By an approximation argument, we may assume that $K$ has a continuous curvature function. E.g. Lemma 3.3 in Livshyts \cite{livshyts} implies \begin{align}\begin{split}
    \mu_1(tK, B_2^n) &= \int_{S^{n-1}}h_{B_2^n}(u) d\sigma_{\mu, tK}(u) \\ &= \int_{S^{n-1}}|u| d\sigma_{\mu, tK}(u).
\end{split}
\end{align}

Let us now use a version of Parseval's formula on the sphere, Lemma 8.8 in Koldobsky \cite{koldbook}, which states that if $E, F$ are origin-symmetric bodies in $\mathbb{R}^n$ such that the support function of $E$ is infinitely smooth on the sphere and the curvature function of $F$ exists and is continuous on the sphere, then \begin{align}
    \int_{S^{n-1}}\widehat{h_E}(\xi)\widehat{f_F}(\xi) d\xi = (2\pi)^n\int_{S^{n-1}}h_E(\xi) f_F(\xi)d\xi.
\end{align}                                    

By (2.5) and (2.8) (see also Livshyts \cite{livshyts}) \begin{align}
    \begin{split}
        \int_{S^{n-1}}|u|d\sigma_{\mu, tK}(u) &= \frac{1}{(2\pi)^n}\int_{S^{n-1}}\widehat{(|x|)}(u)d\widehat{\sigma_{\mu, tK}}(u) \\ &= -\frac{\pi}{n(2\pi)^n}\int_{S^{n-1}}\widehat{(|x|)}(u)p_{\mu, K}(u,t) du
    \end{split}
\end{align}

 Thus, by (2.7), (2.9), and Lemma 2.8, \begin{align*}\begin{split}
     \int_{0}^{1}\mu_1(tK, B_2^n) dt &= \frac{\pi (2\pi)^n \Gamma\left(\frac{n+1}{2}\right)}{n(2\pi)^n \pi^{\frac{n+1}{2}}}\int_{0}^{1}\int_{S^{n-1}}|u|^{-n-1}p_{\mu, K}(u,t) du dt \\ &= \frac{\Gamma\left(\frac{n+1}{2}\right)}{n \pi^{\frac{n-1}{2}}} \int_{S^{n-1}}\int_{0}^{1}p_{\mu, K}(u,t) dt du \\ &= \frac{\Gamma\left(\frac{n+1}{2}\right)}{n \pi^{\frac{n-1}{2}}} \int_{S^{n-1}}P_{\mu, K}(u) du \\ &= \frac{1}{n\omega_{n-1}}\int_{S^{n-1}}P_{\mu, K}(u) du.
     \end{split}
 \end{align*}
 \end{proof}

\end{subsection}
\end{section}

\begin{section}{Sections Bounded by Projections}

\begin{subsection}{Proof of Theorem 1.2.} 
$(a) $ We show the inequalities $|K| \le \frac{R}{r}|L|$ and $|K| \le cL_K^{\frac{1}{2}}n^{\frac{3}{4}}\left(\frac{R}{r}\right)^{\frac{n}{2n-1}}|L|$ separately, beginning with the first. Let us write \begin{align}
    |K \cap \xi^{\perp}| = \frac{1}{n-1}\int_{S^{n-1} \cap \xi^{\perp}}\rho_K^{n-1}(\theta) d\theta,
\end{align} again by polar coordinates (or Lemma 2.7). From Koldobsky \cite{koldbook}, identity (2.22), we have the formula \begin{align}
    \int_{G_{n,k}}\left(\int_{S^{n-1}\cap H}f(\xi) d\xi\right) d\nu_{n,k}(H) &= \frac{|S^{k-1}|}{|S^{n-1}|}\int_{S^{n-1}}f(\xi) d\xi,
\end{align} valid for all $f$ continuous on the sphere. 

Thus, by (3.1) and (3.2) and the fact that $\rho_K \le R$, \begin{align}\begin{split}
    |K| &= \frac{1}{n}\int_{S^{n-1}}\rho_K^n(\theta) d\theta \\ &\le \frac{R}{n}\int_{S^{n-1}}\rho_K^{n-1}(\theta) d\theta \\ &= \frac{R}{n|S^{n-2}|} \int_{S^{n-1}}\left(\int_{S^{n-1} \cap \xi^{\perp}}\rho_K^{n-1}(\theta) d\theta \right) d\xi \\ &= \frac{R}{n\omega_{n-1}}\int_{S^{n-1}}|K \cap \xi^{\perp}| d\xi.
    \end{split}
\end{align}

By (3.3) and Cauchy's formula for surface area (2.6), we write, using the assumption $|K \cap \theta^{\perp}| \le |L| \theta^{\perp}|$ for all $\theta \in S^{n-1}$, \begin{align*}\begin{split}
    |K| &\le \frac{R}{n\omega_{n-1}}\int_{S^{n-1}}|L|\theta^{\perp}| d\theta \\ &= \frac{R}{n}|\partial L|.
    \end{split}
\end{align*}

Since $rB_2^n \subseteq L$, we get, following an argument used by Ball \cite{ballisoperimetric},
\begin{align}\begin{split}
    |\partial L| &= \liminf_{\varepsilon\to 0}\frac{|L+\varepsilon rB_2^n|-|L|}{r\varepsilon} \\
    &\le \liminf_{\varepsilon \to 0}\frac{|L(1+\varepsilon)|-|L|}{r\varepsilon} \\ &= \liminf_{\varepsilon\to 0}|L| \frac{(1+\varepsilon)^n-1}{r\varepsilon} = \frac{n}{r}|L|.
    \end{split}
\end{align} Therefore $|K| \le \frac{R}{r}|L|,$ as desired.

Let us now show $|K| \le cL_K^{\frac{1}{2}}n^{\frac{3}{4}}\left(\frac{R}{r}\right)^{\frac{n}{2n-1}}|L|$ by comparing the minimal (respectively maximal) value of the sections of $K$ and the minimal (respectively maximal) value of the projections of $L$. Defining the parallel section function $A_{K,\theta}(t) = |K \cap \{\theta^{\perp}+t\theta\}|$, Fubini's theorem gives us \begin{align*}
    |K| = \int_{-R}^{R}A_{K, \theta}(t) dt.
\end{align*} As $K$ is origin-symmetric, Brunn's theorem, see e.g. Koldobsky \cite{koldbook}, tells us that $A_{K,\theta}^{\frac{1}{n-1}}$ is concave, even, and hence achieves its maximum at $t = 0$. Thus $|K| \le 2R|K \cap \theta^{\perp}|$, and as $|K \cap \theta|\le |L|\theta^{\perp}|$ for all $\theta \in S^{n-1}$, we conclude \begin{align}\begin{split}
    |K| &\le 2R\min_{\theta \in S^{n-1}}|K \cap \theta^{\perp}| \\ &\le 2R\min_{\theta \in S^{n-1}}|L|\theta^{\perp}| \\ &\lesssim R\sqrt{n}|L|^{\frac{n-1}{n}},
\end{split}
\end{align} where the last line is a result of Ball \cite{ballshadows}. On the other hand, using the inequality $|K|^{\frac{n-1}{n}} \lesssim L_K\max_{\theta \in S^{n-1}}|K \cap \theta^{\perp}|$ from Milman and Pajor \cite{milmanpajor}, we see that \begin{align}\begin{split}
    |K|^{\frac{n-1}{n}} &\lesssim L_K\max_{S^{n-1}}|L|\theta^{\perp}| \\ &\le L_K |\partial L| \\ &\le \frac{L_K n|L|}{r},
\end{split}
\end{align} by (3.4). If we multiply (3.5) and (3.6) and observe that $L_K^{\frac{1}{n}} \lesssim 1$ (which follows from $L_K \lesssim \sqrt{n}$), we arrive at \begin{align*}
    |K| \le cL_K^{\frac{1}{2}}n^{\frac{3}{4}}\left(\frac{R}{r}\right)^{\frac{n}{2n-1}}|L|,
\end{align*} an improvement to the $\frac{R}{r}$ factor when $\frac{R}{r}$ is sufficiently large.

$(b) $ Let $g$ be the density of $\mu$. Integrating in polar coordinates, we get \begin{align}
    \begin{split}
        \mu(K) = \int_{S^{n-1}}\left(\int_{0}^{\rho_K(\theta)}g(r\theta)r^{n-1}dr\right)d\theta
    \end{split}
\end{align} and \begin{align}
    \mu_{n-1}(K \cap \xi^{\perp}) = \int_{S^{n-1} \cap \xi^{\perp}}\left(\int_{0}^{\rho_K(\theta)}g(r\theta)r^{n-2} dr\right) d\theta.
\end{align} Therefore, by (3.7), (3.2), and (3.8), \begin{align}
\begin{split}
    \mu(K) &\le R\int_{S^{n-1}}\left(\int_{0}^{\rho_K(\theta)}g(r\theta)r^{n-2}dr\right)d\theta \\ &= \frac{R}{|S^{n-2}|}\int_{S^{n-1}}\int_{S^{n-1} \cap \xi^{\perp}}\left(\int_{0}^{\rho_K(\theta)}g(r\theta)r^{n-2}dr\right) d\theta d\xi \\ &= \frac{R}{(n-1)\omega_{n-1}}\int_{S^{n-1}}\mu_{n-1}(K \cap \xi^{\perp})d\xi \\ &\le \frac{R}{(n-1)\omega_{n-1}}\int_{S^{n-1}}P_{\mu, L}(\theta) d\theta.
    \end{split}
\end{align} By Proposition 2.9, \begin{align}
    \frac{1}{n\omega_{n-1}}\int_{S^{n-1}}P_{\mu, L}(\theta)d\theta = \int_{0}^{1}\mu_1(tL, B_2) dt.
\end{align} 

Since $rB_2^n \subseteq L$, we note $\mu_1(tL, B_2) \le \frac{1}{r}\mu_1(tL, L) = \frac{1}{r}\mu(tL)'.$ Thus, by (3.9) and (3.10),\begin{align*}
    \mu(K) \le \frac{R}{r\left(1-\frac{1}{n}\right)}\mu(L).
\end{align*}

\begin{remark}
Observe that in the first line of (3.9), we bound by $R$ inside the second integral. This explains the differing results between the Lebesgue and general case. Clearly Theorem 1.2a is sharp for $K = L = RB_2^n = rB_2^n$. 

But Theorem 1.2b is also sharp. To see this, consider the measure $\mu$ with density $|x|^p$ for some $p>0$. Let $K = RB_2^n$ for some fixed $R$, and choose $r$ such that $\mu_{n-1}(K \cap \theta^{\perp}) = P_{\mu, L}(\theta)$ for all $\theta \in S^{n-1}$, where $L = rB_2^n$. Then \begin{align*}
\frac{r\mu(K)}{R\mu(L)} = \frac{p+n-1}{p+n}\frac{1}{1-\frac{1}{n}} \to \frac{1}{1-\frac{1}{n}}
\end{align*} as $p \to \infty$, demonstrating sharpness.
\end{remark}

\end{subsection}

\begin{subsection}{$k-$Dimensional Variant.}
Let us recall that for a convex body $L$, the mean width of $L$ is defined as \begin{align*}
w(L) = \int_{S^{n-1}} h_L(\theta) d\sigma(\theta),
\end{align*} where $\sigma(\theta)$ is the Haar probability measure on $S^{n-1}$. Aleksandrov's inequality (see e.g. Giannopoulos and Koldobsky \cite{giankold} and Schneider \cite{schneiderbook}) states that for any convex body $K$, \begin{align}
\left(\frac{1}{\omega_k} \int_{G_{n,k}} |K|H| d\nu_{n,k}(H) \right)^{\frac{1}{k}} \le \omega(K).
\end{align} Without requiring knowledge of the inradius of $L$, the proof of Theorem 1.2a can be adapted to give us a result in terms of the mean width of $L$ that holds for sections and projections of arbitrary dimension. 
\begin{proposition}
Let $K, L$ be convex bodies in $\mathbb{R}^n$ with $K \subseteq RB_2^n$. If, for some $1\le k \le n-1$, \begin{align*}
|K \cap H| \le |L|H|
\end{align*} for all $H \in G_{n,k}$, then $$|K| \le \omega_n R^{n-k}\omega(L)^k.$$
\end{proposition}
\begin{proof}
By (3.2), \begin{align*}
|K| &= \frac{1}{n}\int_{S^{n-1}}\rho_K^n(\theta) d\theta \\ &\le \frac{R^{n-k}}{n}\int_{S^{n-1}}\rho_K^{k}(\theta)d\theta \\ &= \frac{|S^{n-1}|R^{n-k}}{n|S^{k-1}|}\int_{G_{n,k}}\left(\int_{S^{n-1} \cap H} \rho_K^k(\theta) d\theta\right) d\nu_{n,k}(H) \\ &= \frac{\omega_n R^{n-k}}{\omega_k}\int_{G_{n,k}}\left(\frac{1}{k}\int_{S^{n-1} \cap H}\rho_K^k(\theta) d\theta\right) d\nu_{n,k}(H) \\ &= \frac{\omega_n R^{n-k}}{\omega_k}\int_{G_{n,k}}|K \cap H| d\nu_{n,k}(H).
\end{align*} Since $|K \cap H| \le |L|H|$ for all $H \in G_{n,k}$, we deduce from Aleksandrov's inequality (3.11) that \begin{align*}
|K| &\le \omega_n R^{n-k} \left(\frac{1}{\omega_k}\int_{G_{n,k}}|L|H| d\nu_{n,k}(H)\right) \\ &\le \omega_n R^{n-k}w(L)^k.
\end{align*}
\end{proof}
\begin{remark}
Proposition 3.1 is sharp, via considering $K = L = RB_2^n$.
\end{remark}
\end{subsection}

\end{section}

\begin{section}{Estimates for Measures with $p-$Concave, $\frac{1}{p}-$Homogeneous Densities}

We now turn to discussing generalizations of Giannopoulos and Koldobsky's affirmative answer to Milman's question. Below we prove Theorem 1.4.

\begin{proof}[Proof of Theorem 1.4]
Let $K, L$ be stipulated as in the statement of the theorem. 

Set $q = \frac{1}{n+ \frac{1}{p}}.$ Since $g$ is $\frac{1}{p}-$homogeneous, $\mu$ is a $(n+\frac{1}{p}) = \frac{1}{q}$-homogeneous measure and so \begin{align}\mu(tK)^{1-q} = t^{\frac{1}{q}-1}\mu(K)^{1-q}\end{align} for $t > 0$.
By (4.1), Lemma 2.6, and Proposition 2.9, we write \begin{align}\begin{split}
    \mu(K)^{1-q}\mu(B_2^n)^{q} &= \left(\frac{1}{q}\int_{0}^{1}t^{\frac{1}{q}-1}dt\right) \mu(K)^{1-q}\mu(B_2^n)^q  \\ &= \int_{0}^{1}\frac{1}{q}\mu(tK)^{1-q}\mu(B_2^n)^q dt \\
    &\le \int_{0}^{1}\mu_1(tK, {B_2}^n)dt \\ &= \frac{1}{n\omega_{n-1}}\int_{S^{n-1}}P_{\mu, K}(u)du,
\end{split}
\end{align}

 By assumption, $P_{\mu, K}(u) \le \mu_{n-1}(L \cap u^{\perp})+\varepsilon$, and so we have \begin{align}
     \int_{S^{n-1}}P_{\mu, K}(u) du \le \int_{S^{n-1}}\mu_{n-1}(L \cap u^{\perp}) du + |S^{n-1}|\varepsilon.
 \end{align} By Lemma 2.7, applied first with $L \cap u^{\perp}$ and then with $L$, \begin{align}
     \mu_{n-1}(L \cap u^{\perp}) = \frac{q}{1-q}\int_{S^{n-1} \cap u^{\perp}}\rho_L^{\frac{1}{q}-1}(\theta) g(\theta) d\theta,
 \end{align} and \begin{align}
    \mu(L) = q \int_{S^{n-1}}\rho_L^{\frac{1}{q}}(\theta) g(\theta) d\theta.
\end{align}

Therefore, by (4.4) and (3.2), \begin{align*} \int_{S^{n-1}}\mu_{n-1}(L \cap u^{\perp}) du &= 
    \frac{q}{1-q}\int_{S^{n-1}}\left( \int_{S^{n-1} \cap {\theta^{\perp}}} \rho_L^{\frac{1}{q}-1}(u) g(u) du\right) d\theta \\
    &= \frac{q|S^{n-2}|}{(1-q)}\int_{S^{n-1}}\rho_L^{\frac{1}{q}-1}(\theta) g(\theta) d\theta \\ &= \frac{q(n-1)\omega_{n-1}}{(1-q)}\int_{S^{n-1}}\rho_L^{\frac{1}{q}-1}(\theta) g(\theta)d\theta.
\end{align*} Hence, \begin{align}
    \frac{1}{n\omega_{n-1}}\int_{S^{n-1}}\mu_{n-1}(L \cap u^{\perp}) du = \frac{q\left(1-\frac{1}{n}\right)}{(1-q)}\int_{S^{n-1}}\rho_L^{\frac{1}{q}-1}(\theta) g(\theta) d\theta.
\end{align}

We apply Holder's inequality and (4.5) to get \begin{align}\begin{split}
    \int_{S^{n-1}}\rho_L^{\frac{1}{q}-1}(\theta) g(\theta) d\theta &\le \left(\int_{S^{n-1}}1^{\frac{1}{q}} g(\theta)d\theta\right)^{q} \left(\int_{S^{n-1}}\left(\rho_L^{\frac{1}{q}-1}(\theta)\right)^{\frac{\frac{1}{q}}{\frac{1}{q}-1}}g(\theta)d\theta\right)^{\frac{\frac{1}{q}-1}{\frac{1}{q}}} \\ &= \left(\int_{S^{n-1}}g(\theta) d\theta\right)^q \frac{\mu(L)^{1-q}}{q^{1-q}}.
    \end{split}
\end{align}

As a special case of Lemma 2.7, we have \begin{align*}
    \int_{S^{n-1}}g(\theta) d\theta = \frac{\mu(B_2^n)}{q},
\end{align*} and hence (4.7) becomes \begin{align}
    \int_{S^{n-1}}\rho_K^{\frac{1}{q}-1}(\theta) g(\theta) d\theta \le \frac{\mu(B_2^n)^q \mu(L)^{1-q}}{q}.
\end{align}

Therefore, by (4.2), (4.3), (4.6), and (4.8), \begin{align*}
    \mu(K)^{1-q}\mu(B_2^n)^q \le \left(\frac{1-\frac{1}{n}}{1-q}\right) \mu(L)^{1-q}\mu(B_2^n)^q + \frac{\omega_n}{\omega_{n-1}}\varepsilon,
\end{align*} which finishes the proof.
 
\end{proof}

\begin{remark}
If the assumption of Theorem 1.4 is replaced by the weaker condition \begin{align}\int_{S^{n-1}}P_{\mu, K}(\theta) d\theta \le \int_{S^{n-1}}\mu_{n-1}(L \cap \theta^{\perp}) d\theta,\end{align} then we still get \begin{align*}
\mu(K) \le \left(\frac{1-\frac{1}{n}}{1-q}\right)^{\frac{1}{1-q}}\mu(L).
\end{align*} For any $\mu$, equality can be achieved with $K = B_2^n$ and $L$ a ball chosen to give equality in $(4.9)$.
\end{remark}

\end{section}

\begin{section}{Estimates for $q-$Concave Measures}
In Theorem 1.4, we assumed that our measure $\mu$ had a density $g$ that was both $p-$concave and $\frac{1}{p}-$homogeneous, hence $\mu$ was $q-$concave and $\frac{1}{q}-$homogeneous. In this section, we will drop the condition of homogeneity and prove Theorem 5.1, stated below:

\begin{theorem}
Let $\mu$ be a $q-$concave measure on $\mathbb{R}^n$ with bounded continuous density $g$ for some $q>0$. Assume that $K$ is an origin-symmetric convex body in $\mathbb{R}^n$ and $L$ is a star body in $\mathbb{R}^n$ such that \begin{align*}
    P_{\mu, K}(\theta) \le \mu_{n-1}(L \cap \theta^{\perp})
\end{align*} for all $\theta \in S^{n-1}$. Then, for any fixed parameter $r>0$, we have \begin{align*}
    \mu(K)^{1-q}\mu(rB_2^n)^q \le r\omega_n^{\frac{1}{n}} \norm{g}_{\infty}^{\frac{1}{n}}  \mu(L)^{\frac{n-1}{n}} + \frac{1}{q}\mu(K).
\end{align*}
\end{theorem}
Our proof will require Theorem 3.8/Corollary 3.9 from Livshyts \cite{livshyts}:
\begin{lemma}
\textit{Let $\mu$ be a $q-$concave measure on $\mathbb{R}^n$. For measurable $E, F$, we have the inequality \begin{align*}
    \mu_1(E,F) \ge \mu_1(E,E) + \frac{\mu(F)^q - \mu(E)^q}{q\mu(E)^{q-1}}.
\end{align*}}
\end{lemma}

The crux of our theorem, however, will be the following inequality from Theorem 1.2 of Dann, Paouris, and Pivovarov \cite{DPP}:
\begin{proposition}[Dann, Paouris, Pivovarov]
\textit{Let $1\le k\le n-1$ and $f$ be a nonnegative, bounded, integrable function on $\mathbb{R}^n$. Then \begin{align*}
    \int_{G_{n,k}}\frac{\left(\int_{E}f(x)dx\right)^n}{\norm{f|E}_{\infty}^{n-k}}d\nu_{n,k}(E) \le \frac{\omega_k^n}{\omega_n^k}\left(\int_{\mathbb{R}^n}f(x)dx\right)^k.
\end{align*}
}
\end{proposition}
Here $\norm{f|E}_{\infty}$ is the $L^{\infty}-$norm of $f$ restricted to the $k-$dimensional subspace $E$. Note that $\norm{f|E}_{\infty} \le \norm{f}_{\infty}$ for $f$ continuous.

\begin{proof}[Proof of Theorem 5.1]
Since $\mu$ is $q-$concave, for all measurable $A, B$ and $\lambda \in [0,1]$, \begin{align*}
    \mu(\lambda A + (1-\lambda)B) \ge \left(\lambda \mu(A)^{q}+(1-\lambda)\mu(B)^q\right)^{\frac{1}{q}}.
\end{align*} If $B = \varnothing,$ this becomes $\lambda^{\frac{1}{q}}\mu(A) \le \mu(\lambda A)$ for $\lambda \in [0,1]$, and thus \begin{align}
\begin{split}
    \mu(K)^{1-q}\mu(rB_2^n)^q &= \left(\frac{1}{q}\int_{0}^{1}t^{\frac{1}{q}-1}dt\right)\mu(K)^{1-q}\mu(rB_2^n)^q \\ &\le \int_{0}^{1}\frac{1}{q}\mu(tK)^{1-q}\mu(rB_2^n)^q dt.
    \end{split}
\end{align}

From Lemma 5.2, \begin{align}\begin{split}
    \frac{1}{q}\mu(tK)^{1-q}\mu(rB_2^n)^q &\le \mu_1(tK, rB_2^n) + \frac{1}{q}\mu(tK) - \mu_1(tK, tK) \\ &= r\mu_1(tK, B_2^n) + \frac{1}{q}\mu(tK) - \mu_1(tK, tK).
    \end{split}
\end{align} We observe that \begin{align}
    \mu_1(tK, tK) = t \mu(tK, K) = t\mu(tK)',
\end{align} and so by (5.1), (5.2), (5.3), and an integration by parts we have \begin{align}
\begin{split}
    \mu(K)^{1-q}\mu(rB_2^n)^q &\le r\int_{0}^{1}\mu_1(tK, B_2^n) dt + \frac{1}{q}\int_{0}^{1}\mu(tK) dt - \int_{0}^{1}td\mu(tK) \\ &= r\int_{0}^{1}\mu_1(tK, B_2^n) dt + \left(\frac{1}{q}+1\right)\int_{0}^{1}\mu(tK) dt - \mu(K) \\ &\le r\int_{0}^{1}\mu_1(tK, B_2^n) dt + \left(\frac{1}{q}+1\right)\mu(K) - \mu(K) \\ &= r\int_{0}^{1}\mu_1(tK, B_2^n) dt + \frac{1}{q}\mu(K).
    \end{split}\end{align} Hence, by Proposition 2.9, \begin{align}
    \mu(K)^{1-q}\mu(rB_2^n)^q &\le \frac{r}{n\omega_{n-1}}\int_{S^{n-1}}P_{\mu, K}(u)du + \frac{1}{q}\mu(K).
\end{align}
Since $P_{\mu, K}(u) \le \mu_{n-1}(L \cap u^{\perp}),$ we seek to bound \begin{align*}
    \frac{1}{n\omega_{n-1}}\int_{S^{n-1}}\mu_{n-1}(L \cap \theta^{\perp}) d\theta = \frac{\omega_n}{\omega_{n-1}}\int_{S^{n-1}}\mu_{n-1}(L \cap u^{\perp}) d\sigma(u).
\end{align*} By Jensen's inequality and Proposition 5.3 for $k=n-1$, \begin{align}\begin{split}
    \frac{\omega_n}{\omega_{n-1}}\int_{S^{n-1}}\mu_{n-1}(L \cap u^{\perp}) d\sigma(u) &\le \frac{\omega_n}{\omega_{n-1}}\left(\int_{S^{n-1}}\mu_{n-1}(L \cap u^{\perp})^n d\sigma(u)\right)^{\frac{1}{n}} \\ &= \frac{\omega_n}{\omega_{n-1}}\left(\int_{S^{n-1}}\left(\int_{u^{\perp}}g(x)\chi_L(x) d\lambda_{n-1}(x)\right)^nd\sigma(u)\right)^{\frac{1}{n}} \\ &\le \frac{\omega_n}{\omega_{n-1}}\left(\norm{g}_{\infty} \frac{\omega_{n-1}^n}{\omega_n^{n-1}}\right)^{\frac{1}{n}}\mu(L)^{\frac{n-1}{n}} \\ &= \omega_n^{\frac{1}{n}}  \norm{g}_{\infty}^{\frac{1}{n}} \mu(L)^{\frac{n-1}{n}}.
    \end{split}
\end{align}

By (5.5) and (5.6), we thus have \begin{align*}
    \mu(K)^{1-q}\mu(rB_2^n)^q \le r\omega_n^{\frac{1}{n}}\norm{g}_{\infty}^{\frac{1}{n}} \mu(L)^{\frac{n-1}{n}} + \frac{1}{q}\mu(K).
\end{align*}
\end{proof}

\begin{corollary}
Let $\mu, K, L, r$ be as in Theorem 5.1. If $\mu(K) \le \left(\frac{q }{q+1}\right)^{\frac{1}{q}}\mu(rB_2^n)$, then \begin{align*}
    \mu(K) \le r \omega_n^{\frac{1}{n}} \norm{g}_{\infty}^{\frac{1}{n}} \mu(L)^{\frac{n-1}{n}}.
\end{align*}
\end{corollary}
\begin{proof}
The condition $\mu(K) \le \left(\frac{q}{q+1}\right)^{\frac{1}{q}}\mu(rB_2^n)$ implies \begin{align*}
    \mu(K) \le \mu(K)^{1-q}\mu(rB_2^n)^q - \frac{1}{q}\mu(K).
\end{align*} We conclude by Theorem 5.1.
\end{proof}

\end{section}

\begin{section}{Estimates for Log-concave Measures. }
Let us recall that a measure is called log-concave if for all measurable $K, L$ and $\lambda \in [0,1]$ we have \begin{align*}
    \mu(\lambda K + (1-\lambda) L) \ge \mu(K)^{\lambda}\mu(L)^{1-\lambda}.
\end{align*} A function is called log-concave if its logarithm is concave, and by the Pr$\acute{\text{e}}$kopa-Leindler inequality (see Artstein-Avidan, Giannopoulos, and Milman \cite{AGM}), measures with log-concave densities are also log-concave. 

A function $g$ is called ray-decreasing if $g(tx) \le g(x)$ for all $t \in [0,1]$ and $x \in \mathbb{R}^n$. In this section, we prove the following theorem:
\begin{theorem}
Let $\mu$ be a log-concave measure with continuous ray-decreasing density $g$.  Assume that $K$ is an origin-symmetric convex body in $\mathbb{R}^n$ and $L$ is a star body in $\mathbb{R}^n$ such that $$P_{\mu, K}(\theta) \le \mu_{n-1}(L \cap \theta^{\perp})$$ for all $\theta \in S^{n-1}$. Let $r>0$ be a fixed parameter.
\begin{enumerate}[(a) ]
\item If $\frac{1}{e}\mu(rB_2^n) \le \mu(K) < \mu(rB_2^n)$, then $$\mu(K) \log \frac{\mu(rB_2^n)}{\mu(K)} \le r\omega_n^{\frac{1}{n}}\norm{g}_{\infty}^{\frac{1}{n}} \mu(L)^{\frac{n-1}{n}}.$$
\item If $\mu(K) \le \frac{1}{e}\mu(rB_2^n)$, then \begin{align*}
    \mu(K) \le \left(\frac{er^n\omega_n\norm{g}_{\infty}}{\mu(rB_2^n)}\right)^{\frac{1}{n-1}}\mu(L)
\end{align*}
\end{enumerate}

\end{theorem}

We now prove some lemmas needed for the proof of Theorem 6.1. 

\begin{lemma}
Let $\mu$ be a measure with a ray-decreasing density $g$. If $t \in [0,1]$ and $K$ is measurable, $\mu(tK) \ge t^n\mu(K)$.
\end{lemma}
\begin{proof}
By a change of variables and the fact that $g$ is ray-decreasing, \begin{align*}
    \mu(tK) = \int_{tK}g(y)dy = t^n \int_{K}g(ty) dy \ge t^n \int_{K}g(y) dy = t^n \mu(K).
\end{align*}
\end{proof}

\begin{lemma}
For a measure $\mu$ with continuous ray-decreasing density $g$, \begin{align*}
    \lim_{x\to\infty} \frac{\mu_1(xB_2^n, B_2^n)}{\mu(xB_2^n)} = 0, \lim_{x\to 0} \frac{\mu_1(xB_2^n, B_2^n)}{\mu(xB_2^n)} = \infty.
\end{align*}
\end{lemma}
\begin{proof} 
We write \begin{align*} 
\frac{\mu_1(xB_2^n, B_2^n)}{\mu(xB_2^n)} &= \frac{x^{n-1}\int_{S^{n-1}}g(xy) dy}{x^n \int_{B_2^n}g(xy) dy},
\end{align*} where the numerator comes from (2.7). Therefore \begin{align*}
\frac{\mu_1(xB_2^n, B_2^n)}{\mu(xB_2^n)} &= \frac{\int_{S^{n-1}} g(xy) dy }{x \int_{0}^{1}\int_{S^{n-1}} g(txy)t^{n-1} dy dt} \\ &\le \frac{\int_{S^{n-1}}g(xy) dy}{x \int_{0}^{1}\int_{S^{n-1}}g(xy)t^{n-1} dy dt } \\ &= \frac{n}{x}.
\end{align*} For $x < 1$, we also have \begin{align*}
\frac{\mu_1(xB_2^n, B_2^n)}{\mu(xB_2^n)} &\ge \frac{\int_{S^{n-1}} g(y) dy}{x \omega_n \norm{g}_{\infty}}.
\end{align*} The lemma follows.
\end{proof}

Finally, we have the following analog of Lemma 5.2 for log-concave measures, again from Livshyts \cite{livshyts}.
\begin{lemma}
For a log-concave measure $\mu$ and measurable $E, F$, we have the inequality \begin{align*}
    \mu_1(E, F) \ge \mu_1(E, E) + \mu(E) \log \frac{\mu(F)}{\mu(E)}.\end{align*}
\end{lemma}

We now give the proof of our theorem:
\begin{proof}[Proof of Theorem 6.1]
$(a) $ Given $x > 0$ to be chosen later, by Lemma 6.4 we have \begin{align*}\begin{split}
    x\mu_1(tK, B_2^n) &= \mu_1(tK, xB_2^n) \\ &\ge \mu_1(tK, tK) + \mu(tK) \log \frac{\mu(xB_2^n)}{\mu(tK)} \\ &= t \mu(tK)' + \mu(tK) \log \frac{\mu(xB_2^n)}{\mu(tK)}.
    \end{split}
\end{align*} If we integrate both sides of the equation and use integration by parts on the $t\mu(tK)'$ term, we arrive at \begin{align}
    \begin{split}
        \mu(K) &\le x\int_{0}^{1}\mu_1(tK, B_2^n)dt + \int_{0}^{1}\mu(tK)\log \frac{e\mu(tK)}{\mu(xB_2^n)} dt.
    \end{split}
\end{align}
By the proof of Theorem 5.1, the condition $P_{\mu, K}(\theta) \le \mu_{n-1}(L \cap \theta^{\perp})$ for all $\theta \in S^{n-1}$ tells us that $\int_{0}^{1}\mu_1(tK, B_2^n) dt \le \omega_n^{\frac{1}{n}}\norm{g}_{\infty}^{\frac{1}{n}}\mu(L)^{\frac{n-1}{n}}.$ Thus (6.1) becomes \begin{align}
    \mu(K) \le x \omega_n^{\frac{1}{n}}\norm{g}_{\infty}^{\frac{1}{n}}\mu(L)^{\frac{n-1}{n}} + \int_{0}^{1}\mu(tK) \log e\mu(tK) dt + \left(\int_{0}^{1}\mu(tK) dt\right)\log \frac{1}{\mu(xB_2^n)}.
\end{align}

By Lemma 6.3 and continuity, we can choose our $x>0$ such that \begin{align}
    \frac{\omega_n^{\frac{1}{n}}\norm{g}_{\infty}^{\frac{1}{n}}\mu(L)^{\frac{n-1}{n}}}{\int_{0}^{1}\mu(tK) dt} = \frac{\mu_1(xB_2^n, B_2^n)}{\mu(xB_2^n)}.
\end{align} 

Then, once more by Lemma 6.4, \begin{align}
    r\frac{\mu_1(xB_2^n, B_2^n)}{\mu(xB_2^n)} &\ge x \frac{\mu_1(xB_2^n,B_2^n)}{\mu(xB_2^n)} + \log \frac{\mu(rB_2^n)}{\mu(xB_2^n)},
\end{align} and so by (6.3) and (6.4) we conclude \begin{align}
    \begin{split}
        \log \frac{1}{\mu(xB_2^n)} \le \frac{\omega_n^{\frac{1}{n}}\norm{g}_{\infty}^{\frac{1}{n}}\mu(L)^{\frac{n-1}{n}}}{\int_{0}^{1}\mu(tK)dt}(r-x) - \log \mu(rB_2^n).
    \end{split}
\end{align}

By (6.2), (6.5), and Jensen's inequality, \begin{align*}
\begin{split}
    \mu(K) &\le r\omega_n^{\frac{1}{n}}\norm{g}_{\infty}^{\frac{1}{n}}\mu(L)^{\frac{n-1}{n}} + \int_{0}^{1}\mu(tK) \log \frac{e \mu(tK)}{\mu(rB_2^n)} dt \\ &= r\omega_n^{\frac{1}{n}}\norm{g}_{\infty}^{\frac{1}{n}}\mu(L)^{\frac{n-1}{n}} + \int_{0}^{1} \log\left(\left(\frac{e\mu(tK)}{\mu(rB_2^n)}\right)^{\mu(tK)}\right) dt \\ &\le r\omega_n^{\frac{1}{n}}\norm{g}_{\infty}^{\frac{1}{n}}\mu(L)^{\frac{n-1}{n}} + \log \left(\int_{0}^{1}\left(\frac{e\mu(tK)}{\mu(rB_2^n)}\right)^{\mu(tK)} dt \right) \\ &\le r\omega_n^{\frac{1}{n}}\norm{g}_{\infty}^{\frac{1}{n}}\mu(L)^{\frac{n-1}{n}} + \log \max\left(1, \left(\frac{e\mu(K)}{\mu(rB_2^n)}\right)^{\mu(K)}\right).
    \end{split}
\end{align*}

Since by assumption $\mu(K) \ge \frac{1}{e}\mu(rB_2^n)$, we have \begin{align*}
    \mu(K) \le r\omega_n^{\frac{1}{n}}\norm{g}_{\infty}^{\frac{1}{n}}\mu(L)^{\frac{n-1}{n}} + \mu(K) + \mu(K) \log \frac{\mu(K)}{\mu(rB_2^n)},
\end{align*} and so $$\mu(K) \log \frac{\mu(rB_2^n)}{\mu(K)} \le r\omega_n^{\frac{1}{n}}\norm{g}_{\infty}^{\frac{1}{n}}\mu(L)^{\frac{n-1}{n}}$$ as desired.

$(b)$ Since $\mu(K) \le \frac{1}{e}\mu(rB_2^n)$, for every $t \in [0,1]$ there exists $f(t) \in [0,1]$ such that $\mu(rf(t)B_2^n) = e\mu(tK).$ By the same argument that gave us (6.1), we write \begin{align} \begin{split}
\mu(K) &\le r\int_{0}^{1}f(t)\mu(tK, B_2^n) dt + \mu(tK)\log \frac{e\mu(tK)}{\mu(rf(t)B_2^n)}dt  \\ &= r\int_{0}^{1}f(t)\mu(tK, B_2^n) dt.
\end{split}
\end{align} From Lemma 6.2, $f(t)^n \mu(rB_2^n) \le e\mu(tK) \le e\mu(K),$ and so \begin{align} f(t) &\le \left(\frac{e \mu(K)}{\mu(rB_2^n)}\right)^{\frac{1}{n}}.
\end{align}
Recalling that $\int_{0}^{1}\mu(tK, B_2^n)dt \le \omega_n^{\frac{1}{n}}\norm{g}_{\infty}^{\frac{1}{n}}\mu(L)^{\frac{n-1}{n}}$, and applying (6.6) and (6.7) implies \begin{align*} \mu(K)^{\frac{n-1}{n}} &\le r\left(\frac{e\omega_{n}\norm{g}_{\infty}}{\mu(rB_2^n)}\right)^{\frac{1}{n}}\mu(L)^{\frac{n-1}{n}}
\end{align*} and therefore $$\mu(K) \le \left(\frac{er^n\omega_n\norm{g}_{\infty}}{\mu(rB_2^n)}\right)^{\frac{1}{n-1}}\mu(L).$$
\end{proof}
\begin{remark}
The two inequalities of Theorem 6.1 are sharp at least up to factors of ${e^{\frac{1}{n}}}$ and ${e^{\frac{1}{n-1}}}$ respectively, which tend to $1$ as $n \to \infty$. To see both, let $\mu = \lambda$, which is log-concave by the Brunn-Minkowski inequality, and take $K = L$ to be the ball with measure $\frac{1}{e}\lambda(rB_2^n),$ where $r>0$ is fixed. In part (a), we have $\mu(K) \log \frac{\mu(rB_2^n)}{\mu(K)} = \frac{1}{e}\lambda(rB_2^n)$ and $r\omega_n^{\frac{1}{n}}\norm{g}_{\infty}^{\frac{1}{n}} \mu(L)^{\frac{n-1}{n}} = \frac{1}{e^{\frac{n-1}{n}}} \lambda(rB_2^n)$, while in part (b), we have $\mu(K) = \frac{1}{e}\lambda(rB_2^n)$ and $\left(\frac{er^n\omega_n\norm{g}_{\infty}}{\mu(rB_2^n)}\right)^{\frac{1}{n-1}}\mu(L) = \frac{e^{\frac{1}{n-1}}}{e}\lambda(rB_2^n)$.
\end{remark}

\end{section}

\begin{section}{Appendix}
As promised, we provide the proofs of two results that were stated previously.

\textbf{Lemma 2.2. }\textit{Let $\mu$ be a Borel measure on $\mathbb{R}^n$ and $\mu_{e}$ be its extension to a distribution with degree of homogeneity $-(n+1)$. Then \begin{align*}
    \widehat{\mu_e}(\theta) = -\frac{\pi}{2} \int_{S^{n-1}}|\langle u, \theta \rangle| d\mu(u)\end{align*}for all $\theta \in S^{n-1}.$}

\begin{proof}
The proof we follow is that presented in Theorem 1 of Koldobsky, Ryabogin, and Zvavitch \cite{KRZ}. Let $\phi$ be an even Schwartz function with $0 \notin \text{supp}(\hat{\phi})$. By (2.4), \begin{align}
\begin{split}
    \langle \widehat{\mu_e}, \phi \rangle &=  \langle \mu_e, \hat{\phi} \rangle \\ &= \int_{S^{n-1}}\int_{0}^{\infty}r^{-2}\hat{\phi}(r\theta) dr d\mu(\theta).
    \end{split}
\end{align} By Lemma 2.11 in Koldobsky \cite{koldbook}, we have that, for fixed $z$, the Fourier transform of $r \to g(z,r):= \int_{\langle x,z \rangle  = r}{\phi}(x)dx$ is the map $r \to \hat{\phi}(rz)$, as seen by the computation \begin{align*}
\widehat{g(z,\cdot)}(r) &= \int_{\mathbb{R}}e^{-irt}\left(\int_{\langle x,z\rangle=t} \phi(x) dx \right) dt \\ &= \int_{\mathbb{R}^n}\phi(x)e^{-ir\langle x,z\rangle} dx \\ &= \hat{\phi}(rz).
\end{align*} Moreover, by the next lemma, the Fourier transform of the distribution $|t|$ on $\mathbb{R}$ is $-2t^{-2}$, and so, following the proof of Lemma 1 in Koldobsky \cite{inverseformula}, the inner integral equals \begin{align}
\begin{split}
    \frac{1}{2}\langle r^{-2}, \hat{\phi}(r\theta) \rangle &= -\frac{1}{4}\left\langle \widehat{(|t|)}(\xi), \widehat{g(\theta, t)}(\xi)\right \rangle \\ &= -\frac{2\pi}{4}\langle |\xi|, g(\theta, \xi) \rangle \\ &= -\frac{\pi}{2}\int_{\mathbb{R}}|\xi|\left(\int_{\langle x,\theta\rangle=\xi}\phi(x) dx\right) d\xi \\ &= -\frac{\pi}{2} \int_{\mathbb{R}^n}|\langle x,\theta\rangle|\phi(x) dx.
    \end{split}
\end{align} Combining (7.1), (7.2) gives us \begin{align*}
\begin{split}
    \langle \widehat{\mu_e}, \phi \rangle = \left \langle -\frac{\pi}{2}\int_{S^{n-1}}|\langle u, \theta \rangle| d\mu(u), \phi \right\rangle.
    \end{split}
\end{align*} Let $\rho$ be the distribution \begin{align}
    \rho(\theta) = \widehat{\mu_e}(\theta) + \frac{\pi}{2}\int_{S^{n-1}}|\langle u, \theta \rangle | d\mu(u).
\end{align} Then $\text{supp}{(\hat{\rho})} \subseteq \{0\}$ and so $\hat{\rho}$ is a linear combinations of distributional derivatives of the Dirac mass at $0$. Applying the Fourier transform to $\hat{\rho}$, we see that $\rho$ is therefore a polynomial. Since $\rho$ is even and has degree of homogeneity $1$, it must be identically zero and our proof is complete by (7.3).
\end{proof}

\textbf{Lemma 2.8. }\textit{In $\mathbb{R}^n$, the Fourier transform of the distribution $|x|$ is equal to \begin{align*}
    \widehat{(|x|)}(\xi) = -\frac{(2\pi)^n\Gamma\left(\frac{n+1}{2}\right)}{\pi^{\frac{n+1}{2}}}|\xi|^{-n-1}.
\end{align*}}

\begin{proof}
By analytic continuation, it suffices to show show that for any $q \in (-n,0)$ the locally integrable function $|x|^{q}$ has Fourier transform \begin{align*}
    \frac{(2\pi)^n2^{q}\Gamma\left(\frac{n+q}{2}\right)}{\pi^{\frac{n}{2}}\Gamma\left(-\frac{q}{2}\right)}|\xi|^{-n-q}.
\end{align*} We show this by adapting the proof in the case $n=1$ (see e.g. Lemma 2.23 in Koldobsky \cite{koldbook}) accordingly. Let $\phi$ be an even Schwartz function. By a change of variables in the definition of the Gamma function, \begin{align}
    |x|^q = \frac{2^{q/2+1}}{\Gamma(-q/2)}\int_{0}^{\infty}z^{-1-q}e^{-z^2|x|^2/2}dz.
\end{align} It is well-known that \begin{align}
    \widehat{(e^{-|x|^2/2})}(z) = (2\pi)^{\frac{n}{2}}e^{-|z|^2/2}.
\end{align} Hence, by (7.4), (7.5) and Parseval's Theorem, \begin{align*}
    \langle \widehat{(|x|^q)}, \phi \rangle &= \langle |x|^q, \hat{\phi}\rangle \\ &= \frac{2^{q/2+1}}{\Gamma(-q/2)}\int_{\mathbb{R}^n} \left(\int_{0}^{\infty}z^{-1-q}e^{-z^2|x|^2/2} dz\right)\hat{\phi}(x) dx \\ &= \frac{2^{q/2+1}}{\Gamma(-q/2)}\int_{0}^{\infty}\left(\int_{\mathbb{R}^n}e^{-z^2|x|^2/2}\hat{\phi}(x) dx\right)z^{-1-q}dz \\ &= \frac{2^{q/2+1}(2\pi)^n}{\Gamma(-q/2)}\int_{0}^{\infty}\left(\int_{\mathbb{R}^n}e^{-\frac{|x|^2}{2z^2}} \phi(x) dx\right)z^{-n-q-1}dz \\ &= \left \langle \frac{2^{q/2+1}(2\pi)^n}{\Gamma(-q/2)}\int_{0}^{\infty}e^{-\frac{|x|^2}{2z^2}}z^{-n-q-1}dz, \phi(x) \right \rangle.
\end{align*} Making the substitution $u = \frac{1}{z}$ in the integral on the left hand side of the inner product, our conclusion follows by another application of (7.4). 
\end{proof}

\end{section}

\bibliographystyle{alpha}
\bibliography{main.bib}

\end{document}